\documentclass[a4paper,egregdoesnotlikesansseriftitles,final]{scrartcl}
\usepackage[utf8]{inputenc}
\usepackage[british]{babel}
\usepackage[T1]{fontenc}
\usepackage{lmodern}
\usepackage{microtype}
\usepackage{amsfonts,amsthm,amssymb,amsmath}
\usepackage{mathtools}
\usepackage{letterswitharrows}
\usepackage{xcolor}
\usepackage{graphicx}
\usepackage[colorlinks,unicode,bookmarks,pdfusetitle,linkcolor={red!60!black},citecolor={green!60!black},urlcolor={blue!60!black}]{hyperref}
\usepackage[capitalize,nameinlink]{cleveref}
\crefname{enumi}{}{}
\crefname{property}{property}{properties}
\crefname{LEM}{Lemma}{Lemmas}
\crefname{THM}{Theorem}{Theorems}
\usepackage{enumitem}
\usepackage{thmtools, thm-restate} 
\usepackage{cite}
\usepackage[abbrev]{amsrefs}
\usepackage{doi}
\renewcommand{\PrintDOI}[1]{\doi{#1}}

\newtheorem{THM}{Theorem}
\newtheorem{LEM}[THM]{Lemma}

\theoremstyle{definition}

\newcommand{\abs}[1]{\lvert#1\rvert}
\newcommand{\menge}[1]{\left\{#1\right\}}

\renewcommand{\phi}{\varphi}

\newcommand{\join}{\lor}
\newcommand{\meet}{\land}
\newcommand{\sub}{\subseteq}
\newcommand{\sm}{\smallsetminus}
\newcommand{\cP}{\mathcal{P}}

\title{A canonical tree-of-tangles theorem for structurally submodular separation systems}
\author{Christian Elbracht\and Jakob Kneip}
\date{16th April 2021}

\begin{document}

\maketitle


\begin{abstract}
	We show that every structurally submodular separation system admits a canonical tree set which distinguishes its tangles.
\end{abstract}

\section{Introduction}
The concept of~\emph{tangles} has its origins in the graph minors project of Robertson and Seymour~\cite{GMX}, where tangles were introduced as a unifying framework with which to describe and study highly cohesive substructures in graphs. A central theorem in their work is a~\emph{tree-of-tangles theorem}, which roughly says that the tangles of a graph give rise to a tree-decomposition of that graph with each tangle in a different bag.

Since their inception the theory of tangles has seen a number of advancements: it has been discovered~\cites{AbstractSepSys,ProfilesNew} that the notion of tangles can be formulated more abstractly and does not require an underlying graph structure, making it applicable to a wider range of combinatorial structures. A~\emph{separation system} in this abstract set-up is an axiomatisation of the properties of separations of well-known structures such as graphs or matroids: it is a set, whose elements we call~\emph{separations}, that is equipped with a partial ordering in terms of which all the properties important for tangle theory, such as `nested', `orientation', or `consistency', can be expressed~\cite{AbstractSepSys}.

This higher level of abstraction, and of stripping away the superfluous information about the underlying graph, have facilitated a number of cleaner proofs and stronger results. A recent result~\cite{ProfilesNew} of Diestel, Hundertmark, and Lemanczyk extends the tree-of-tangles theorem of Robertson and Seymour~\cite{GMX} to tangles outside graph theory by finding a~\emph{tree set}, a set of pairwise nested separations, and in addition to this achieves a significant strengthening: the tree set found in~\cite{ProfilesNew} can be built~canonically. The latter means that the construction of the tree set can be carried out using only invariants of the given combinatorial structure. Having a canonical way of constructing the tree set is desirable, for instance, for reproducibility of results when implementing an algorithm for this construction: the canonicity guarantees that the algorithm will construct the same tree set regardless of how the separation system and tangles to be distinguished are presented to it as input.

Establishing canonical tree-of-tangles theorems has been a long-standing goal in tangle theory, since the original proof in~\cite{GMX} relied on a technique that is unable to produce canonical results. A first breakthrough towards this goal was achieved in~\cite{CDHH13CanonicalAlg}, which managed to establish such a canonical theorem for tangles in graphs. With a similar overall strategy Diestel, Hundertmark and Lemanczyk~\cite{ProfilesNew} could then extend this canonical result to arbitrary separation systems and the most general class of tangles, which are called~\emph{profiles}.

A central ingredient in~\cite{ProfilesNew} is an~\emph{order function} on the separations considered, similar to the order~$ \abs{A\cap B} $ of a separation~$ (A,B) $ of a graph that was already used in~\cite{GMX} and~\cite{CDHH13CanonicalAlg}. In this setting one then considers the separation system~$ \vS_k $ of all separations of order less than~$ k $ and studies its tangles. In analogy to the function~$ {(A,B)\mapsto\abs{A\cap B}} $ from graphs this order function is usually assumed to be submodular. This submodularity of the order function has a structural effect on the separation system~$ \vS_k $: for any two separations in~$ \vS_k $ at least one of their pairwise join and meet (which are separations given by opposite `corners' of these separations) again lies in~$ \vS_k $.

Later Diestel, Erde, and Weißauer~\cite{AbstractTangles} showed that the latter structural condition by itself is already sufficiently strong for proving tree-of-tangles theorems: tangle theory can be meaningfully studied without the hitherto usual assumption of a submodular order function, further widening its applicability. If a separation system has this structural property but not necessarily a submodular order function then it is~\emph{structurally submodular} or simply~\emph{submodular} if the context is clear. The tree-of-tangles theorem established in~\cite{AbstractTangles} then reads as follows:

\begin{THM}[\cite{AbstractTangles}*{Theorem 6}]\label{thm:Daniel}
	Let~$ \vS $ be a structurally submodular separation system and~$ \cP $ a set of profiles of~$ S $. Then~$ \vS $ contains a tree set~$ N $ that distinguishes~$ \cP $. 
 \end{THM}

\noindent Here, a tree set~$ N $~\emph{distinguishes}~$ \cP $ if each pair of profiles in~$ \cP  $ lies on different sides of some separation in~$ N $. For formal definitions see~\cref{sec:defs}.

This~\cref{thm:Daniel} is even more widely applicable than the result of~\cite{ProfilesNew}, but has one major downside: it does not yield canonicity, since the proof in~\cite{AbstractTangles} chooses certain separations arbitrarily.

In this note we present a new proof, which establishes the following canonical version of~\cref{thm:Daniel}:
\begin{THM}\label{thm:canonical_old}
	Let~$ \vS $ be a structurally submodular separation system and~$ \cP $ a set of profiles of~$ S $. Then there is a nested set~$ N=N(\vS,\cP)\sub S $ which distinguishes~$ \cP $. This~$ N(\vS,\cP) $ can be chosen canonically: if~$ \phi\colon\vS\to\vS' $ is an isomorphism of separation systems and~$ \cP'\coloneqq\menge{\phi(P)\mid P\in\cP} $ then~$ \phi(N(\vS,\cP))=N(\vS',\cP') $.
\end{THM}
There are some technical subtleties in the formulation of~\cref{thm:canonical_old} due to the fact that neither the profile property nor structural submodularity need be preserved by isomorphisms of separation systems. To avoid these difficulties, we obtain~\cref{thm:canonical_old} by first establishing the following more general but somewhat more technical result, which slightly weakens the definitions of submodularity and profiles in order to make them compatible with such isomorphisms:

\begin{THM}\label{thm:canonical}
	Let~$ \vS $ be a separation system and $\cP$ a collection of consistent orientations of $\vS$ such that $\vS$ is $\cP$-submodular. Then there is a nested set~$ N=N(\vS,\cP)\sub S $ which distinguishes~$ \cP $. This~$ N(\vS,\cP) $ can be chosen canonically: if~$ \phi\colon\vS\to\vS' $ is an isomorphism of separation systems and~$ \cP'\coloneqq\menge{\phi(P)\mid P\in\cP} $ then~$ \phi(N(\vS,\cP))=N(\vS',\cP') $.
\end{THM}

\cref{thm:canonical_old} is then an immediate corollary of~\cref{thm:canonical}.

This paper is structured as follows: in~\cref{sec:defs} we recall the relevant definitions of separation systems and profiles. In~\cref{sec:P-submodularity} we introduce the new definition required for~\cref{thm:canonical} and show that~\cref{thm:canonical_old} indeed is an immediate corollary of~\cref{thm:canonical}. In~\cref{sec:proof} we prove~\cref{thm:canonical_old,thm:canonical}.

\section{Separation Systems and Profiles}\label{sec:defs}

For a full introduction to abstract tangle theory and its terminology and notation we refer the reader to~\cite{AbstractSepSys,AbstractTangles}. In the remainder of this section we offer a brief introduction of only those terms and notation of tangle theory that are relevant to this paper.

A~\emph{separation system}~$ \vS=(\vS,\le,^*) $ is a poset~$ (\vS,\le) $ together with an order-reversing involution~$ ^* $ on~$ \vS $. Given an element~$ \vs $ of~$ \vS $ we denote its~\emph{inverse}, its image under~$ ^* $, by~$ \sv=(\vs)^* $. For~$ ^* $ to be order-reversing then means that~$ \vr\le\vs $ if and only if~$ \rv\ge\sv $ for all~$ \vr,\vs\in\vS $.

We refer to the elements of~$ \vS $ as~\emph{oriented separations}. Given such an oriented separation~$ \vs $ we call~$ \{\vs,\sv\} $ its underlying~\emph{unoriented separation} and denote it by~$ s $. Conversely, given an unoriented separation~$ s $, we call~$ \vs $ and~$ \sv $ the two~\emph{orientations} of~$ s=\{\vs,\sv\} $.

If the context is clear we often refer to both oriented and unoriented separations simply as~\emph{separations}. Moreover, when no confusion is possible, we may informally use terms defined for unoriented separations for oriented separations as well and vice-versa. For a set~$ \vS $ of oriented separations we write~$ S $ for the set of all~$ s $ with~$ \vs\in\vS $. Conversely, if~$ S $ is a set of unoriented separations, we write~$ \vS $ for the set of all~$ \vs $ and~$ \sv $ with~$ s\in S $.

Two separations~$ r $ and~$ s $ are~\emph{nested} if they admit orientations~$ \vr $ and~$ \vs $ with~$ \vr\le\vs $. If~$ r $ and~$ s $ are nested we also call all their respective orientations nested. Two separations that are not nested are said to~\emph{cross}. Note that two oriented separations~$ \vr $ and~$ \vs $ need not be themselves comparable in order to be nested: we may have for instance that~$ \vr\le\sv $. A set of separations is~\emph{nested} if its elements are pairwise nested.

We say that two oriented separations~$ \vr $ and~$ \vs $~\emph{point towards each other} if~$ \vr\le\sv $, and that they~\emph{point away from each other} if~$ \rv\le\vs $. Thus~$ \vr $ and~$ \vs $ are nested if and only if they are either comparable, or point towards each other, or point away from each other.

If~$ \vU=(\vU,\le,^*) $ is a separation system whose poset is a lattice with pairwise join and meet operations~$ \join $ and~$ \meet $ we call~$ (\vU,\le,^*,\join,\meet) $ a~\emph{universe of separations}. Given~$ r $ and~$ s $ in~$ U $ we call all separations of the form~$ \vr\join\vs $ as well as their inverses and underlying separations~\emph{corner separations} of~$ r $ and~$ s $. Note that in universes of separations DeMorgan's rule holds:
\[ (\vr\join\vs)^*=(\rv\meet\sv) \]

\noindent For a universe~$ \vU=(\vU,\le,^*,\join,\meet) $ of separations and a separation system~$ \vS\sub\vU $ we say that~$ \vS $ is~\emph{(structurally) submodular in $\vU$} if for all~$ \vr $ and~$ \vs $ in~$ \vS $ at least one of~$ \vr\join\vs $ and~$ \vr\meet\vs $ is also contained in~$ \vS $.

An~\emph{orientation} of a set~$ S $ of unoriented separations is a set~$ O\sub\vS $ containing exactly one orientation~$ \vs $ or~$ \sv $ of each~$ s=\{\vs,\sv\} $ in~$ S $. We call~$ O $~\emph{consistent} if there are no~$ \vr $ and~$ \vs $ in~$ O $ with~$ \rv\le\vs $ and~$ r\ne s $.

Given a universe~$ \vU $ of separations and a separation system~$ \vS\sub\vU $, a~\emph{profile} of~$ S $ is a consistent orientation of~$ S $ such that
\[ \forall \,\vr,\vs\in P \colon (\rv\meet\sv)\notin P \tag{P}\,.\label[property]{property:P} \]
This~\cref{property:P} is commonly referred to as the~\emph{profile property}.

A separation~$ s $~\emph{distinguishes} two profiles~$ P $ and~$ Q $ of~$ S $ if there is an orientation~$ \vs $ of~$ s $ with~$ \vs\in P $ and~$ \sv\in Q $. 
A set~$ N $~\emph{distinguishes} a set~$ \cP $ of profiles of~$ S $ if every pair of distinct profiles in~$ \cP $ is distinguished by some~$ s $ in~$ N $.

A set~$ N\sub S $ of unoriented separations is a~\emph{tree set} if~$ N $ is nested and there are no~$ r\neq s\in N$ with orientations $\vr$ and $\vs$ such that~$ \vr\le\vs $ and~$ \vr\le\sv $. Note that if~$ N $ is nested and has the property that each separation in~$ N $ distinguishes some pair of profiles of~$ S $, then~$ N $ is a tree set.

Finally, an~\emph{isomorphism} of separation systems~$ \vS $ and~$ \vS' $ is a bijective map~$ \phi\colon\vS\to\vS' $ such that~$ \phi(\sv)=(\phi(\vs))^* $ for all~$ \vs $ in~$ \vS $ and furthermore~$ \vr\le\vs $ if and only if~$ \phi(\vr)\le\phi(\vs) $ for all~$ \vr $ and~$ \vs $ in~$ \vS $.

The most prominent occurrence of separation systems is in graphs: given a graph~${G=(V,E)}$, a separation of~$ G $ is a pair~$ (A,B) $ of vertex sets with~$ A\cup B=V $ such that there is no edge from~$ A\sm B $ to~$ B\sm A $. With involution~$ (A,B)^*=(B,A) $ and partial order~$ (A,B)\le(C,D) $ if and only if~$ A\sub C $ and~$ B\supseteq D $, the set of all separations of~$ G $ becomes a separation system -- in fact, a universe of separations.

Of special importance in the theory of graph separations are the separation systems~$ \vS_k $ of all separations~$ (A,B) $ of~$ G $ with~$ \abs{A\cap B}<k $. Each such~$ \vS_k $ constitutes a structurally submodular separation system by our definition. The~$ k $-tangle of a graph, whose introduction and study by Robertson and Seymour in~\cite{GMX} was the inception of tangle theory, is then a special type of consistent orientation of~$ S_k $. Each such~$ k $-tangle is a profile of~$ S_k $ in our sense. Profiles, both in graphs and in abstract separation systems, were first rigorously studied in~\cite{ProfilesNew}: that work also contains a plethora of examples and applications of profiles, including an example of a profile in a graph that is not a~$ k $-tangle~(\cite{ProfilesNew}*{Example~7}).

The separation systems considered in~\cite{ProfilesNew}, although a generalisation of graph separations, still come with some strong structural assumptions that are closely modelled on graph separations. The first paper to study profiles in universes of separations as defined here, without any additional structural assumptions, was~\cite{AbstractTangles}. That work, too, contains some examples of profiles in separation systems, as well as multiple applications of the theory of tangles and profiles both inside and outside of graph theory. The applications presented in~\cite{AbstractTangles} include profiles in matroids and using tangles for cluster analysis, as well as working with more exotic types of separations such as clique separations in graphs or circle separations outside of graphs.

Both~\cite{ProfilesNew} and~\cite{AbstractTangles} provide their own version of~\cref{thm:canonical_old}. The version given in~\cite{ProfilesNew}, like our~\cref{thm:canonical_old}, constructs a~\emph{canonical} tree set distinguishing a given set of profiles, but does so by leveraging a much stronger set of assumptions. Conversely, the version \cref{thm:Daniel} of~\cref{thm:canonical_old} given in~\cite{AbstractTangles} does not yield canonicity.

\section{Submodularity with respect to a set of profiles}\label{sec:P-submodularity}

The first hurdle to overcome when aiming for a canonical version of~\cref{thm:Daniel} is to pin down what exactly `canonical' ought to mean. At first glance this is obvious: the construction of the nested set~$ N $ shall use only invariants of~$ \vS $ and~$ \cP $, that is, properties which are preserved by isomorphisms of separation systems. This approach, however, runs into a subtle difficulty: the definitions of both structural submodularity and profiles depend on~$ \vS $ being embedded into an ambient universe of separations, whose existence~\cref{thm:Daniel} implicitly assumes. An isomorphism~$ \phi\colon\vS\to\vS' $ of separation systems, though, need not preserve such an embedding, which leads to the undesirable situation that a construction isomorphic to that of~$ N $ in~$ \vS $ could be carried out in~$ \vS' $, even though~\cref{thm:canonical_old} may not be directly applicable to~$ \vS' $ due to differences in their embeddings into ambient lattices.

To make our canonical version of~\cref{thm:Daniel} as widely applicable as possible, and to keep the definition of canonicity as straightforward and clean as possible, we must therefore tweak the assumptions of structural submodularity and profiles of~$ S $ in such a way that they no longer depend on any embedding into a universe of separations, and are themselves invariants of isomorphisms between separation systems. This is made possible by the following observation: the proof of~\cref{thm:Daniel} makes use of the assumptions that~$ \vS $ is submodular and~$ \cP $ a set of profiles solely to deduce that whenever some~$ \vr $ and~$ \vs $ in~$ \vS $ distinguish some two profiles in~$ \cP $, then either their meet or their join (as provided by the ambient lattice) is contained in~$ \vS $ and likewise distinguishes that pair of profiles.

For our canonical~\cref{thm:canonical} we will thus eliminate the need for an ambient universe by asking of~$ \vS $ and~$ \cP $ that they have this property, with the meets or joins now being taken directly in the poset~$ \vS $. Expressed solely in terms of~$ \vS $ and~$ \cP $, this `richness' property is then preserved by isomorphisms of separation system, independently of any embeddings into lattice structures. This solves the minor problem in the formulation of~\cref{thm:canonical_old} of~$ \vS' $ not meeting the assumption of the theorem despite being isomorphic to a separation system which does.~\cref{thm:canonical_old} is then obtained as a corollary of~\cref{thm:canonical}.

Let~$ \vS $ be a separation system and~$ \cP $ a set of consistent orientations of~$ \vS $. Given a set~$ M\sub\vS $ of oriented separations, an element~$ \vr\in\vS $ is an~\emph{infimum} of~$ M $ in~$ \vS $ if~$ \vr\le\vs $ for each~$ \vs\in M $ and additionally~$ \vr\ge\vt $ whenever~$ \vt\in\vS $ is such that~$ \vt\le\vs $ for all~$ \vs\in M $. Dually, an element~$ \vr\in\vS $ is a~\emph{supremum} of~$ M $ in~$ \vS $ if~$ \vr\ge\vs $ for each~$ \vs\in M $ and additionally~$ \vr\le\vt $ whenever~$ \vt\in\vS $ is such that~$ \vt\ge\vs $ for all~$ \vs\in M $. In general a set~$ M\sub\vS $ need not have such an infimum or supremum in~$ \vS $.

Given two separations~$ \vr $ and~$ \vs $ in~$ \vS $ we denote the infimum and supremum of~$ \menge{\vr,\vs} $ in~$ \vS $ by~$ \vr\meet\vs $ and~$ \vr\join\vs $, respectively, if those exist. Observe that~$ (\vr\join\vs)^\ast=\rv\meet\sv $.

If~$ \vr $ and~$ \vs $ have a supremum~$ \vr\join\vs $ in~$ \vS $, and every~$ P\in\cP $ containing both~$ \vr $ and~$ \vs $ also contains~$ \vr\join\vs $, then we call~$ \vr\join\vs $ a~\emph{$ \cP $-join} of~$ \vr $ and~$ \vs $ in~$ \vS $. Dually we call~$ \vr\meet\vs $ a~\emph{$ \cP $-meet} of~$ \vr $ and~$ \vs $ if~$ (\vr\meet\vs)^\ast\in P $ for each~$ P\in\cP $ containing both~$ \rv $ and~$ \sv $.

Finally, we say that~$ \vS $ is~\emph{$ \cP $-submodular} if every two crossing separations~$ \vr $ and~$ \vs $ in~$ \vS $ have a~$ \cP $-join or~$ \cP $-meet in~$ \vS $. For the remainder of this work we assume~$ \vS $ to be~$ \cP $-submodular.

Observe that~$ \cP $-submodularity is preserved by isomorphisms of separation systems: if~$ \phi\colon\vS\to\vS' $ is an isomorphism and~$ \cP'\coloneqq\menge{\phi(P)\mid P\in\cP} $, then~$ \vS' $ is~$ \cP' $-submodular. Using this notion of submodularity we can therefore meaningfully express canonicity in the context of~\cref{thm:Daniel}.

Note, however, that we are not assuming the elements of~$ \cP $ to be profiles of~$ \vS $: this is precisely because we prove~\cref{thm:canonical} without an ambient lattice structure, which would be necessary to define profiles. Therefore,~\cref{thm:canonical} improves on~\cref{thm:Daniel} not only by offering a canonical way of constructing~$ N $, but also by being applicable to an even larger number of separation systems.

Before getting to the proof of~\cref{thm:canonical} itself, let us demonstrate that it is in fact a strengthening of~\cref{thm:Daniel} by showing that~\cref{thm:canonical} implies~\cref{thm:canonical_old}. The following lemma does just that by proving that in the setting of~\cref{thm:Daniel} the assumptions of~\cref{thm:canonical} are satisfied:

\begin{LEM}\label{lem:Psub_strucsub}
	Let~$ \vS $ be a structurally submodular separation system inside some universe~$ \vU $ of separations and~$ \cP $ a set of profiles of~$ \vS $. Then~$ \vS $ is~$ \cP $-submodular.
\end{LEM}

\begin{proof}
	We must show that any~$ \vr $ and~$ \vs $ in~$ \vS $ have a~$ \cP $-meet or~$ \cP $-join. So let~$ \vr $ and~$ \vs $ in~$ \vS $ be given. 
	Since~$ \vS $ is structurally submodular it contains the infimum~$ \vr\meet\vs $ or the supremum~$ \vr\join\vs $ of~$ \vr $ and~$ \vs $ in~$ \vU $. Let us assume the latter; the other case is dual. Then~$ \vr\join\vs $ is also the supremum of~$ \vr $ and~$ \vs $ as taken in~$ \vS $. Moreover, by the profile property, every~$ P\in\cP $ containing both~$ \vr $ and~$ \vs $ also contains~$ \vr\join\vs $, making this a~$ \cP $-join of~$ \vr $ and~$ \vs $ in~$ \vS $.
\end{proof}

\section{\texorpdfstring{Proof of~\cref{thm:canonical_old,thm:canonical}}{Proof of Theorem 2 and 3}}\label{sec:proof}
In this section we will prove~Theorem~\ref{thm:canonical_old} and~Theorem~\ref{thm:canonical}. To this end, for the remainder of this paper let~$\vS$ be a separation system and let~$\cP$ be a set of consistent orientations of~$\vS$ such that~$\vS$ is~$\cP$-submodular.

A common tool in proving tree-of-tangles theorems is the so-called fish lemma \cite{AbstractSepSys}*{Lemma~3.2}. Since this lemma is usually formulated in the context of a separation system that is contained in a universe of separations, we need to prove our own version of this lemma.

\begin{LEM}[See also \cite{AbstractSepSys}*{Lemma 3.2}]\label{lem:fish}
 	Let~$r,s\in S$ be two crossing separations in~$ S $ and let~$t\in S$ be a separation that is nested with both~$r$ and~$s$. Given orientations~$\vr$ and~$\vs$ of~$r$ and~$s$ such that there exists a supremum~$\vr\join \vs$ of~$\vr$ and~$\vs$ in~$\vS$, the separation~$\vr\join \vs$ is nested with~$t$. The same is true for~$ \vr\meet\vs $.
\end{LEM}

\begin{proof}
 If $t$ has an orientation $\vt$ such that $\vt\le \vr$ or $\vt\le \vs$ then clearly $\vt\le (\vr\join \vs)$. Otherwise, since $r$ and $s$ cross, there must be an orientation $\vt$ of $t$ such that $\vr\le \vt$ and $\vs\le \vt$. Thus, by the fact that $\vr\join \vs$ is the supremum of $\vr$ and $\vs$ in $\vS$, we have that~$(\vr\join \vs)\le \vt$.
\end{proof}

Let us say that a separation~$ \vs\in\vS $ is~\emph{exclusive (for~$ \cP $)} if it lies in exactly one orientation in~$ \cP $. If~$ P\in\cP $ is the orientation containing an exclusive separation~$ \vs $ then we might also say that~$ \vs $ is~\emph{$ P $-exclusive (for~$ \cP $)}. Observe that if~$ \vr $ is~$ P $-exclusive for~$ \cP $, then so is every~$ \vs\in P $ with~$ \vr\le\vs $.

For each~$ P\in\cP $ let~$ M_P $ consist of the maximal elements of the set of all~$ P $-exclusive separations. Equivalently,~$ M_P $ is the set of all maximal elements of~$ P $ that are exclusive for~$ \cP $.

Our strategy for proving~\cref{thm:canonical} will be to canonically pick nested $ P $-exclusive representatives of all orientations~$ P\in\cP $ that contain exclusive separations, then discard from~$ \cP $ and~$ \vS $ all those orientations~$ P $ for whom we selected a representative and all those separations not nested with these representatives, respectively. Iterating this procedure will yield the canonical nested set.

In order for this strategy to work we must ensure that the sets~$ M_P $ are not all empty. Our first lemma addresses this:

\begin{LEM}\label{lem:nonempty}
	If~$ \cP $ and $S$ are non-empty, then some~$ M_P $ is non-empty.
\end{LEM}

In the case of $\vS$ being submodular in some universe of separation $\vU$ and $\cP$ being a set of profiles of $\vS$, the existence of exclusive separations and thus~\cref{lem:nonempty} is actually an immediate consequence of~\cref{thm:Daniel}: if~$ N\sub S $ is a nested set which distinguishes~$ \cP $, and each element of~$ N $ distinguishes some pair of profiles in~$ \cP $, then any maximal element of~$ \vN $ is exclusive for~$ \cP $. In other words, the separations labelling the incoming edges of leaves of the tree associated with~$ N $ are exclusive. (See~\cite{TreeSets} for the precise relationship between nested sets and trees.)

However, since we are working with the more general notion of $\vS$ being $\cP$-submodular, we give an independent proof of~\cref{lem:nonempty}.

\begin{proof}[Proof of~\cref{lem:nonempty}.]
	If~$ \cP $ consists of only one orientation the assertion is trivial since $S$ is non-empty. For~$ \abs{\cP}\ge 2 $ we show the following stronger claim by induction on~$ \abs{\cP} $:
	\begin{center}\em
		If~$ \abs{\cP}\ge 2 $ there is for each~$ P\in\cP $ a separation that is exclusive but not~$ P $-exclusive for~$ \cP $.
	\end{center}
	For the base case~$ \abs{\cP}=2 $ observe that any separation distinguishing the two orientations in~$ \cP $ has two exclusive orientations, one in each element of $\cP$.
	
	Suppose now that~$ \abs{\cP}>2 $ and that the claim holds for all non-singleton proper subsets of~$ \cP $. Let~$ P\in\cP $ be the given fixed orientation and set~$ \cP'\coloneqq\cP\sm\menge{P} $. By the induction hypothesis applied to~$ \cP' $ and an arbitrary orientation in $\cP'$ there is an exclusive separation~$ \vr $ for~$ \cP' $, contained in some~$ Q\in\cP' $. Applying the induction hypothesis again to~$ \cP' $ and~$ Q $ yields another separation~$ \vs $ that is exclusive for~$ \cP' $ and lies in some~$ Q'\in\cP' $ with~$ Q\ne Q' $.
	
	If~$ \vr $ or~$ \vs $ is also exclusive for~$ \cP $ then we are done. So suppose not, that is, suppose we have~$ \vr,\vs\in P $. Then~$ r\ne s $, and hence~$ \vr $ and~$ \vs $ must be incomparable by the consistency of~$ Q $ and~$ Q' $. If~$ \vr\le\sv $ then~$ \sv $ is~$ Q $-exclusive for~$ \cP $. Moreover $\sv\le \vr$ is not possible by the consistency of $P$. Thus we may assume that~$ r $ and~$ s $ cross.
	
	By $\cP$-submodularity of $\vS$, there is a $\cP$-join or a $\cP$-meet of $\vr$ and $\sv$ in $\vS$; by symmetry we may assume that there is a $\cP$-join~$ (\vr\join\sv)\in\vS $. Since~$ \vs $ is~$ Q' $-exclusive we have~$ \sv\in Q $ and hence~$ (\vr\join\sv)\in Q $ by the fact that $\vr\join \sv$ is the $\cP$-join of $\vr$ and $\sv$. From~$ (\vr\join\sv)\ge\vr $ we infer that~$ \vr\join\sv $ is~$ Q $-exclusive for~$ \cP' $. Moreover we cannot have~$ (\vr\join\sv)\in P $: it would be inconsistent with~$ \vs\in P $ as~$ r $ and~$ s $ cross.
	
	Therefore~$ \vr\join\sv $ is exclusive but not~$ P $-exclusive for~$ \cP $.
\end{proof}

We remark that, in the case of a submodular separation system $\vS$ inside a universe of separations and a set $\cP$ of profiles, the stronger assertion used for the induction hypothesis in this proof, too, can be established immediately using~\cref{thm:Daniel}: for~$ \abs{\cP}\ge 2 $ the tree associated with the nested set~$ N\sub S $ distinguishing~$ \cP $ has at least two leaves, and hence some leaf for which the separation labelling its incoming edge does not lie in the fixed profile~$ P $.

Returning to the proof of~\cref{thm:canonical}, let us find a way to canonically pick representatives of those~$ P\in\cP $ with non-empty~$ M_P $ in such a way that these representatives are nested with each other. For the `canonically'-part of this we will make use of the fact that the sets~$ M_P $ themselves are invariants of~$ \cP $ and~$ \vS $. For the nestedness we start by showing that separations from different~$ M_P $'s cannot cross at all:

\begin{LEM}\label{lem:nestedbetween}
	For~$ P\ne P' $ all~$ \vr\in M_P $ and~$ \vs\in M_{P'} $ are pairwise nested.
\end{LEM}

\begin{proof}
	Suppose some~$ \vr\in M_P $ and~$ \vs\in M_{P'} $ cross. By $\cP$-submodularity of $\vS$, there is a $\cP$-join~$ \vr\join\sv $ or a $\cP$-meet~$ \vr\meet\sv $ in~$ \vS $; by symmetry we may suppose that~$ (\vr\join\sv)\in\vS $. Then~$ P $, too, contains this separation since~$ \sv\in P $. But~$ \vr\join\sv $ is also~$ P $-exclusive and strictly larger than~$ \vr $, a contradiction.
\end{proof}

It is possible, however, that the set~$ M_P $ itself is not nested. In fact the elements of~$ M_P $ all cross each other, unless~$ \cP=\menge{P} $: any~$ \vr $ and~$ \vs $ in~$ M_P $ that are nested must point towards each other by maximality. But every other orientation in~$ \cP $ contains both~$ \rv $ and~$ \sv $ and would then be inconsistent. If we want to represent a~$ P\in\cP $ with non-empty~$ M_P $ by an element of~$ M_P $, we are therefore limited to picking at most one element of~$ M_P $. However there is no canonical way of singling out an element of~$ M_P $ to be the representative of~$ P $; we must therefore find another way of choosing an invariant~$ P $-exclusive separation, using~$ M_P $ only as a starting point.

For this we will show that each $ M_P $ has an infimum in $\vS$ and that this infimum is again~$ P $-exclusive:

\begin{LEM}\label{lem:infimum}
	Let~$ P\in\cP $ with~$ M_P\ne\emptyset $ and~$ \cP\ne\menge{P} $ be given. Then~$ M_P $ has an infimum~$ \vs_P $ in the poset~$ \vS $, and~$ \vs_P $ is~$ P $-exclusive for~$ \cP $. Moreover if some~$ t\in S $ is nested with~$ M_P $ then~$ t $ is also nested with~$ s_P $.
\end{LEM}

\begin{proof}
	Fix an enumeration~$ M_P=\menge{\vr_1,\dots,\vr_n} $ and some~$ t\in S $ that is nested with~$ M_P $. We will show by induction on $i$ that there is an infimum  $\vs_i\coloneqq \inf\{\vr_1,\dots,\vr_i\}$ in~$ \vS $; that this infimum is~$ P $-exclusive for~$ \cP $; and that it is nested with~$ t $. This then yields the claim for~$ i=n $.
	
	The case~$ i=1 $ is trivially true, so suppose that~$ i>1 $ and that~$ \vs_{i-1}=\inf \{\vr_1,\dots,\vr_{i-1}\} $ is already known to be the infimum of $\vr_1,\dots, \vr_{i-1}$ in~$ \vS $, that it is~$ P $-exclusive, and that it is nested with~$ t $.
	
	In the case that~$ s_{i-1}=r_i $ we have either~$\vs_{i-1}=\vr_i$ or~$ \vs_{i-1}=\rv_i $. The latter of these is impossible since~$ P $ contains both of the~$ P $-exclusive separations~$ \vs_{i-1} $ and~$ \vr_i $. The former, however, gives that~$ \vr_i $ is the infimum of~$ \vr_1,\dots,\vr_i $ and thus as claimed by~$ \vr_i\in M_P $.
	
	So suppose that~$ s_{i-1}\ne r_i $. Let us first treat the case that~$ \vr_i $ and~$ \vs_{i-1} $ are nested. Clearly the two cannot point away from each other since~$ P $ is consistent. If~$ \vr_i $ and~$ \vs_{i-1} $ are comparable then one of the two is the infimum of~$\vr_i$ and~$\vs_{i-1}$ and thus the infimum of~$\vr_1, \dots, \vr_i$ in~$\vS$. Since both~$\vs_{i-1}$ and~$\vr_i$ are~$P$-exclusive and nested with~$t$, this infimum is thus as claimed. Finally, if~$ \vr_i $ and~$ \vs_{i-1} $ point towards each other, we obtain a contradiction: for then their inverses point away from each other, making every orientation in~$ \cP $ other than~$ P $ inconsistent. Thus if~$ \vr_i $ and~$ \vs_{i-1} $ are nested the induction hypothesis holds for~$ \vs_i $.
	
	Let us now consider the case that~$ \vr_i $ and~$ \vs_{i-1} $ cross. Then there needs to be a $\cP$-join or a $\cP$-meet of $\vr_i$ and $\vs_{i-1}$. However we cannot have a $\cP$-join~$ \vr_i\join\vs_{i-1} $ in~$ \vS $ since this join would be~$ P $-exclusive and strictly larger than~$ \vr_i\in M_P $. Therefore there is a $\cP$-meet~$ (\vr\meet\vs_{i-1})\in\vS $. By consistency we have that~$ \vs_i\in P $. Every orientation in~$ \cP $ other than~$ P $ contains~$ \rv_i $ as well as~$ \sv_{i-1} $ and hence~$ \sv_{i} $ by the definition of $\cP$-meet, which shows that~$ \vs_i $ is~$ P $-exclusive. Finally, by~\cref{lem:fish},~$ \vs_i $ is also nested with~$ t $.
\end{proof}

It remains to show that after picking as a representative for each~$ P\in\cP $ with exclusive separations the infimum of~$ M_P $, the set of separations in~$ \vS $ that are nested with all these representatives is still rich enough to distinguish all orientations in~$ \cP $ for which we have not yet picked a representative.

For this let~$ \vS'\sub\vS $ be the system of all those separations that are nested with all~$ M_P $, and let~$ \cP'\sub\cP $ be the set of those orientations~$ Q $ that have empty~$ M_Q $. Our next lemma says that if we restrict ourselves to~$ \vS' $, we can still distinguish~$ \cP' $:

\begin{LEM}\label{lem:restriction}
	The separation system~$ \vS' $ is $\cP'$-submodular and distinguishes~$ \cP' $.
\end{LEM}

\begin{proof}
	The fact that~$ \vS' $ is $\cP'$-submodular is a direct consequence of~\cref{lem:fish}: it implies that for~$ \vr $ and~$ \vs $ in~$ \vS $ any~$ \cP $-meet or~$ \cP $-join of them in~$ \vS $ is contained in~$ \vS' $. Since~$ \cP'\sub\cP $ this~$ \cP $-join or~$ \cP $-meet is in fact a~$ \cP' $-join or~$ \cP' $-meet of~$ \vr $ and~$ \vs $ in~$ \vS' $.
	
	To see that $\vS'$ distinguishes $\cP'$, let~$ Q $ and~$ Q' $ be distinct orientations in~$ \cP' $; we shall show that some~$ \vs'\in\vS' $ distinguishes them. For this choose a separation~$ s\in S $ which distinguishes~$ Q $ and~$ Q' $ and which is nested with~$ M_P $ for as many~$ P\in\cP $ as possible. If~$ s $ is nested with all~$ M_P $ we are done; otherwise there is some~$ P\in\cP $ for which~$ s $ crosses some separation in~$ M_P $. 
	
	So suppose that there is a~$ P\in\cP $ for which~$ s $ is not nested with~$ M_P $. Among all~$ \vs'\in\vS $ which distinguish~$ Q $ and~$ Q' $ and which are nested with each~$ M_{P'} $ with which~$ s $ is nested, pick a minimal~$ \vs' $ with~$ \vs'\in P $. We claim that this~$ \vs' $ is nested with~$ M_P $, contradicting the choice of~$ s $.
	
	To see this, suppose that~$ \vs' $ crosses some~$ \vr\in M_P $. Then there cannot exist a~$\cP$-join of~$\vr$ and~$\vs'$ in~$\vS$ since this join would be a strictly larger~$ P $-exclusive separation than~$ \vr $. Hence there is a~$\cP$-meet~$\vr\meet\vs'$ of~$\vr$ and~$\vs'$ in~$\vS $. By~$ P\notin\menge{Q,Q'} $ we have that both~$ Q $ and~$ Q' $ contain~$ \rv $, and hence this separation $\vr\meet \vs'$ distinguishes~$ Q $ and~$ Q' $ as well: one of the two orientations contains $\vs'$ and thus also $\vr\meet \vs'$ by consistency. The other contains both~$\sv'$ and~$\rv$ and thus also~$(\rv\join \sv')=(\vr\meet \vs')^\ast$ by the fact that~$\vr\meet \vs'$ is the~$\cP$-meet of~$\vr$ and~$\vs'$. 
	However, by~\cref{lem:fish} and~\cref{lem:nestedbetween}, this~$ \vr\meet\vs' $ would be nested with each~$ M_{P'} $ with which~$ s $ was nested, while being strictly smaller than~$ \vs' $, a contradiction.
\end{proof}

If~$ M_P $ is non-empty let us write~$ \vs_P $ for its infimum in~$ \vS $ as in~\cref{lem:infimum}. We are now ready to prove~\cref{thm:canonical} by induction.

\begin{proof}[Proof of~\cref{thm:canonical}.]
	We proceed by induction on~$ \abs{\cP} $. If~$ \abs{\cP}\le 1 $ there is nothing to show, so suppose that~$ \abs{\cP}>1 $ and that the assertion holds for all proper subsets of~$ \cP $. 
	
	Recall that~$ \vS'\sub\vS $ consists of all separations in~$ \vS $ that are nested with all sets~$ M_Q $ and that~$ \cP'\sub\cP $ is the set of all~$ Q\in\cP $ with empty~$ M_Q $. Clearly both~$ \vS' $ and~$ \cP' $ are invariants of~$ \vS $ and~$ \cP $ since the sets~$ M_Q $ themselves are invariants. For each non-empty~$ M_Q $ let~$ \vs_Q $ be its infimum in~$ \vS $ as described in~\cref{lem:infimum}. Then
	\[ N_1\coloneqq\menge{s_Q\mid Q\in\cP\sm\cP'} \]
	is clearly a canonical set. From~\cref{lem:infimum} we further know that~$ N_1 $ distinguishes all orientations in~$ \cP\sm\cP' $ from each other and from each orientation in~$ \cP' $.
	
	By~\cref{lem:nestedbetween} every element of~$ M_P $ is nested with every element of~$ M_{P'} $ for all~$ P\ne P' $. Applying the `moreover'-part of~\cref{lem:infimum} twice thus implies that~$ s_P $ is nested with every element of~$ M_{P'} $ and subsequently with~$ s_{P'} $. Therefore~$ N_1 $ is a nested set. Likewise every separation in~$ \vS' $ is nested with~$ N_1 $.
	
	Let us apply the induction hypothesis to~$ \cP' $ in~$ \vS' $, as made possible by~\cref{lem:nonempty,lem:restriction}, yielding a canonical nested set~$ N_2\sub S' $ which distinguishes~$ \cP' $. Since~$ \vS' $ and~$ \cP' $ themselves are invariants of~$ \vS $ and~$ \cP $ we have that the union~$ N_1\cup N_2 $ is the desired canonical nested set.
\end{proof}

\begin{proof}[Proof of \cref{thm:canonical_old}]
 By \cref{lem:Psub_strucsub}, given a structurally submodular separation system $\vS$ and a set $\cP$ of profiles of $S$, we know that $\vS$ is $\cP$-submodular. Thus \cref{thm:canonical_old} follows from \cref{thm:canonical}.
\end{proof}

\section*{Acknowledgement}
We would like to thank one of the reviewers for suggesting the concept of~$\cP$-submodularity, which led to \cref{thm:canonical}.

\bibliography{collective}

\vspace{0.5cm}
\noindent
\begin{minipage}{\linewidth}
 \raggedright\small
   \textbf{Christian Elbracht},
   \texttt{christian.elbracht@uni-hamburg.de}

   \textbf{Jakob Kneip},
   \texttt{jakob.kneip@uni-hamburg.de}

   Universit\"at Hamburg,
   Bundesstra\ss{}e 55,
   20146 Hamburg, Germany
\end{minipage}
\end{document}